\documentclass[12pt]{amsart}

\usepackage{amsmath,amsthm,amscd,amsfonts,amssymb,epic,eepic,bbm}
\usepackage[pagebackref,colorlinks=true,linkcolor=blue,citecolor=blue]{hyperref}
\usepackage{tikz-cd}

\allowdisplaybreaks
\newtheorem{thm}{Theorem}[section]
\newtheorem{cor}[thm]{Corollary}
\newtheorem{conj}[thm]{Conjecture}
\newtheorem{lem}[thm]{Lemma}

\theoremstyle{remark}
\newtheorem*{rem}{Remark}

\newcounter{remarkscounter}
\newenvironment{remarks}
{\medskip\noindent{\it
Remarks.}\begin{list}{{\rm(\arabic{remarkscounter})}
}{\usecounter{remarkscounter}

\setlength{\labelsep}{\fill} \setlength{\leftmargin}{0pt}
\setlength{\itemindent}{\fill}
\setlength{\labelwidth}{\fill}\setlength{\topsep}{0pt}
\setlength{\listparindent}{0pt}}} {\end{list}}

\numberwithin{equation}{section}
\newcommand{\A}{\mathbb{A}}
\newcommand{\GL}{\mathrm{GL}}

\newcommand{\ZZ}{\mathbb{Z}}

\newcommand{\QQ}{\mathbb{Q}}

\newcommand{\lto}{\longrightarrow}

\newcommand{\OO}{\mathcal{O}}
\newcommand{\CC}{\mathbb{C}}
\newcommand{\RR}{\mathbb{R}}
\newcommand{\GG}{\mathbb{G}}

\newcommand{\gl}{\mathfrak{gl}}

\newcommand{\quash}[1]{}

\theoremstyle{definition}

\renewcommand{\bar}{\overline}

\numberwithin{equation}{subsection}

\newcommand{\LL}{\mathbb{L}}

\begin{document}
\title[Nonabelian Fourier transforms]{Nonabelian Fourier transforms for spherical representations}
\author{Jayce R. Getz}
\address{Department of Mathematics\\
Duke University\\
Durham, NC 27708}
\email{jgetz@math.duke.edu}


\thanks{The author is thankful for partial support provided by NSF grant DMS 1405708.  Any opinions, findings, and conclusions or recommendations expressed in this material are those of the author and do not necessarily reflect the views of the National Science Foundation.}

\begin{abstract}
Braverman and Kahzdan have introduced an influential conjecture on local functional equations for general Langlands $L$-functions.  It is related to
 Lafforgue's equally influential conjectural construction of kernels for functorial transfers.
We formulate and prove a version of Braverman and Kazhdan's conjecture for spherical representations over an archimedean field that is suitable for application to the trace formula.  We then give a global application related to Langlands' beyond endoscopy proposal.  It is motivated by Ng\^o's suggestion that one combine nonabelian Fourier transforms
with the trace formula in order to prove the functional equations of Langlands $L$-functions in general.  

\end{abstract}

\maketitle

\section{Introduction}

Let $G$ be a split connected reductive group over an archimedean field $F$ and let
\begin{equation}\label{r}
r:\widehat{G} \lto \GL_{n}
\end{equation}
be a representation of (the connected component of) its Langlands dual group, which we regard as a connected reductive group over $\CC$.  For simplicity we assume that the neutral component of the kernel of $r$ is trivial (this is the most interesting case anyway).  The local Langlands correspondence is a theorem of Langlands in the archimedean case.  Thus
for every irreducible admissible representation $\pi$ of $G(F)$ one has an irreducible admissible representation $r(\pi)$ of $\GL_n(F)$ that is the 
 transfer of (the $L$-packet of) $\pi$. One defines
$$
\gamma(s,\pi,r,\psi):=\gamma(s,r(\pi),\psi):=\frac{\varepsilon(s,r(\pi),\psi) L(1-s,r(\pi)^\vee)}{L(s,r(\pi))}
$$
for $s\in \CC$ and for any additive character $\psi:F \to \CC^\times$, where the $\varepsilon$-factor on the right is that defined by Godement and Jacquet \cite{GodementJacquetBook}.    Let $f \in C_c^\infty(G(F))$.  In the case where $r$ is the standard representation of $\GL_n$ and the representation $\pi$ is unitary one has an identity of operators
\begin{align} \label{loc-fe-00}
\gamma(s,\pi,r,\psi)\pi|\det|^{\tfrac{n-1}{2}+s}(f)=\pi^{\mathrm{anti}}|\det|^{\tfrac{n-1}{2}+1-s}(\widehat{f})
\end{align}
where
$\pi^{\mathrm{anti}}(g):=\pi(g^{-1})$ and $\widehat{f}$ is the restriction to $\GL_n(F)$ of the $\gl_n(F)$-Fourier transform of $f$ determined by $\psi$ (see \cite[(9.5)]{GodementJacquetBook}).  Strictly speaking $f$ is assumed to be in a space of Gaussian functions in \cite{GodementJacquetBook}, but there is no need to make this precise here.

A conjecture of Braverman and Kazhdan \cite{BK-lifting} states that \eqref{loc-fe-00} is but the first case of a general phenomenon.  To be more precise we need to place an additional assumption on the representation $r$.
Assume that there is a character 
\begin{align} \label{omega}
\omega: G \lto \GG_m
\end{align}
such that, 
if we denote by $\omega^{\vee}$ the dual cocharacter of the center $Z_{\widehat{G}}$ of $\widehat{G}$, then
$$
r \circ \omega^{\vee}=[N]
$$
where $[N]:\GG_m \to \GL_n$ is the 
cocharacter given on points by
$$
x \longmapsto x^NI_n.
$$  
Here $I_n$ is the $n\times n$ identity matrix. 
For complex numbers $s$ we let
$$
\omega_s:=|\omega|^{s/N} \quad \textrm{ and } \quad \pi_s:=\pi \otimes \omega_s.
$$
The quasicharacter $\omega_s$ plays the role of $|\det |^s$ in the case considered by Godement and Jacquet in \cite{GodementJacquetBook}.

Temporarily let $F$ be an arbitrary local field.  Braverman and Kazhdan gave a conjectural construction of a nonabelian Fourier transform 
\begin{align} \label{FT:BK}
\mathcal{F}_{r,\psi}:C_c^\infty(G(F)) &\lto C^\infty(G(F))
\end{align}
such that
\begin{align} \label{lfe}
\gamma(s,\pi,r,\psi)\pi_{s}(f)=\pi^{\mathrm{anti}}_{1-s}(\mathcal{F}_{r,\psi}(f)).
\end{align} 
In the case where $r$ is the standard representation of $G=\GL_n$ one can take $\mathcal{F}_{r,\psi}(f)=|\det|^{\tfrac{n-1}{2}}(|\det|^{\tfrac{1-n}{2}}f)^{\wedge}$.  In the nonarchimedean case L.~Lafforgue has given a spectral definition of $\mathcal{F}_{r,\psi}$ using Paley-Weiner theory under suitable assumptions that are implied by the local Langlands correspondence for $G(F)$ \cite[D\'efinition II.15]{LafforgueJJM}.  
The analytic properties of $\mathcal{F}_{r,\psi}(f)$ (e.g. whether or not it is integrable after a suitable twist by a quasicharacter of $G(F)$) are not obvious from 
 his construction.  

Assume again that $F$ is archimedean, and let $K \leq G(F)$ be a maximal compact subgroup.  In this paper we prove the existence of a transform $\mathcal{F}_{r,\psi}(f)$ 
such that \eqref{lfe} holds for unitary representations $\pi$ provided that $f$ is spherical and lies in a naturally defined subspace $C_c^\infty(G(F)//K,r)$ of $C_c^\infty(G(F)//K)$ depending on $r$.
  We also prove that suitable twists of $\mathcal{F}_{r,\psi}(f)$ by quasicharacters lie in a space of functions for which the trace formula is valid.  

  In \S \ref{sec-main-thm} we define the subspace $C_c^\infty(G(F)//K,r) \leq C_c^\infty(G(F)//K)$. For $0<p\leq 2$ let
$$
\mathcal{S}^p(G(F) // K) \leq C^\infty(G(F)//K)
$$
be the $L^p$-Harish-Chandra Schwartz space; we recall its definition in \S \ref{sec-sp-funcs} below.  The following is the main theorem of this paper:

\begin{thm}\label{thm-lfe} 
Let $f \in C_c^\infty(G(F)//K,r)$ and  $0<p \leq 1$.  There is 
a function
$$
\mathcal{F}_{r,\psi}(f) \in C^\infty(G(F)//K)
$$
such that 
\begin{enumerate}
\item[(a)] one has
$
\mathcal{F}_{r,\psi}(f)\omega_s \in \mathcal{S}^p(G(F)//K)
$
for $\mathrm{Re}(s)$ sufficiently large in a sense depending only on $G$ and $r$, and

\item[(b)] if $\pi$ is unitary and irreducible then 
\begin{align} \label{loc-fe-0}
\gamma(s,\pi,r,\psi)\mathrm{tr}\,\pi_{s}(f)=\mathrm{tr}\,\pi^\vee_{1-s}(\mathcal{F}_{r,\psi}(f))
\end{align}
in the sense of analytic continuation.
\end{enumerate}
\end{thm}

\begin{remarks}

\item A more precise version of (a) is proved in Theorem \ref{thm-lfe2} below.

\item Elements of $\mathcal{S}^p(G(F)//K)$ for $0<p \leq 1$ are in $L^1(G(F)//K)$, so the fact that $\pi$ is unitary implies that the operator $\pi^{\vee}_s(\mathcal{F}_{r,\psi}(f))$ is bounded for $\mathrm{Re}(s)$ sufficiently large.

\item The function $f$ and $\mathcal{F}_{r,\psi}(f)$ are assumed to be spherical, so for $\pi$ unitary and nonspherical  \eqref{loc-fe-0} is just the equality $0=0$.

\item The operator $\pi_s(f)$ is holomorphic as a function of $s$, and $\gamma(s,\pi,r,\psi)$ is meromorphic.  Thus \eqref{loc-fe-0} provides a meromorphic continuation of $\mathrm{tr}\,\pi^\vee_{1-s}(\mathcal{F}_{r,\psi}(f))$ to the complex plane.  

\end{remarks}

\bigskip

Assertion (a) in the theorem is important because the Arthur-Selberg trace formula is valid for functions in (the global version of) $\mathcal{S}^p(G(F)//K)$ for $0<p \leq 1$ 
due to work of Finis, Lapid, and M\"uller \cite{FinisLapidMullerAbsConv,FinisLapid2011,FinisLapidContinuityGeometric}.  One can then use the results of Finis, Lapid, and M\"uller to provide an absolutely convergent expression for the sum of residues of Langlands $L$-functions that Langlands has isolated for study in his ``beyond endoscopy" proposal.  This will be discussed in \S \ref{sec:app} below.

It would be very interesting to extend the results in this paper to nonspherical representations.  Our approach might be applicable in this more general setting provided one can prove a certain analytic result.  More precisely, we use the characterization of the image of $\mathcal{S}^p(G(F)//K)$ under the Fourier transform due to Trombi and Varadarajan \cite{TrombiVaradarajanSpherical}.  A characterization of the image of the nonspherical analogue $\mathcal{S}^p(G(F))$ under the Fourier transform does not seem to be available in the literature, although partial results are known \cite{TrombiHarmonic,TrombiUniform}.  Hopefully the conjectures of Braverman-Kahzdan \cite{BK-lifting},  Lafforgue's work \cite{LafforgueJJM}, and of course the beyond endoscopy proposal of Langlands \cite{LanglandsBE} will provide motivation for giving such a characterization.

\begin{rem} Arthur has characterized the image of the Fourier transform of $\mathcal{S}^2(G(F))$ \cite{ArthurPWActa}, but this does not seem to be the right space from the point of view of any of these proposals to move beyond endoscopy.  It would be interesting to see if the technique of Anker \cite{AnkerSimple} could be used to deduce the image of the Fourier transform on $\mathcal{S}^p(G(F))$ for $0<p<2$ from this case.
\end{rem}

We would be remiss not to recall that the fundamental aim of  \cite{BK-lifting} and \cite{LafforgueJJM} is to provide a  definition of the nonabelian Fourier transform for which a version of Poisson summation is valid.  As explained in loc.~cit., this would lead to a proof of 
the functional equation and meromorphic continuation of the $L$-functions attached to $r$ by Langlands.  For this purpose it would probably be desirable to have a definition of the Fourier transform $\mathcal{F}_{r,\psi}(f)$, which, unlike the approach of this paper and \cite{LafforgueJJM},  does not rely on Paley-Wiener theorems.  
For this we can only point to the hints provided in the works \cite{AltugBEI,BouthierNgoSakellaridis,
Cheng:Ngo:BK,ChengGlobalTF,FLN,Getz4var,Getz:RSMonoid,
GetzHerman,Getz:Liu:BK, WWLiSat,LanglandsTransfert,NgoSums,
SakellaridisSph}.  

We close the introduction by outlining the sections of this paper.  In \S \ref{sec-transfer} we recall the notion of the transfer of a spherical representation.  Preliminaries on the characterization of the image of the spherical Fourier transform due to Trombi and Varadarajan are given in \S \ref{sec-sp-funcs} and in \S \ref{sec-main-thm} we prove Theorem \ref{thm-lfe2}, which immediately implies our main result, Theorem \ref{thm-lfe}. Finally, in \S \ref{sec:app} we give an application of the main theorem to Langlands' beyond endoscopy proposal.

\section{Tori and transfers of representations}
\label{sec-transfer}

Let $F$ be a local field.  In this section we set some notation and recall the notion of the transfer $r(\pi)$ of an irreducible admissible representation $\pi$ of $G(F)$ under some simplifying assumptions.  These assumptions are always true if $\pi$ is spherical.  Before this we give a definition, from \cite{Cheng:Ngo:BK}, of a useful extension $W'$ of the Weyl group $W$ of a split maximal torus $T \leq G$ by a subgroup of $\mathfrak{S}_n$, the symmetric group on $n$ letters.  The whole point of the discussion below is to explain how $r$ induces a 
$W'$-equivariant map
$$
r^{\vee}:T_n \lto T
$$
where $T_n$ is a maximal torus in $\GL_n$.

Let $T \leq G$ be a split maximal torus with Weyl group $W$.  Moreover let $\widehat{T} \leq \widehat{G}$ be the dual torus; its Weyl group in $\widehat{G}$ is isomorphic to $W$ and we denote it by the same letter.  Let $V_r$ be the space of $r$.   We can decompose
\begin{align} \label{decomp}
V_r = \oplus_{i=1}^m V_{\lambda_i}
\end{align}
where the sum is over the nonzero weights $\lambda_1,\dots,\lambda_m \in X^*(\widehat{T})$ of $\widehat{T}$ in $V_r$, $V_{\lambda_i}$ is the $\lambda_i$ weight space, and $\mathrm{dim} V_{\lambda_i}=d_i$.  Fix a basis $A_{i} \subset V_r$ for each $V_{\lambda_i}$ (viewed as a subspace of $V_r$) and let $A=\coprod_{i=1}^mA_i$; this is a basis for 
$V_r$.  This choice of basis gives an embedding $\GG_m^A \to \GL(V_r)$.  
We let $\widehat{T}_n$ be its image and let $\Lambda=X^*(\widehat{T}_n)=\ZZ[A]$. It comes equipped with a $\ZZ$-linear map
\begin{align} \label{basis-map}
\Lambda \lto X^*(\widehat{T})=X_*(T)
\end{align}
given by extending the set theoretic map sending each basis element 
in $A_{i}$ to $\lambda_i$. 
Thus $r$ induces a map
\begin{align} \label{pin}
r:\widehat{T} \lto \widehat{T}_n
\end{align}
 where $\widehat{T}_n \leq \GL_n$ is the maximal torus with character group $\Lambda$.  In fact, upon conjugating the representation $r$ by an element of $\GL_n(\CC)$ we can and do assume that $\widehat{T}_n$ is the standard maximal torus of diagonal matrices. 
 For $F$-algebras $R$ we take
$$
T_n(R)=\Lambda \otimes R^\times.
$$
It is a torus over $F$ with dual torus $\widehat{T}_n$, and by construction there is a morphism
\begin{align} \label{pin:dual}
r^{\vee}:T_n \lto T
\end{align}
over $F$ whose dual is $r$.

The Weyl group $W_n$ of $\GL_n$ can be identified with the set of permutations of $A$ (which we also identify with $\mathfrak{S}_n$) and we let
\begin{align}
\Sigma_{\underline{\lambda}}=\mathfrak{S}_{r_1} \times ...\times \mathfrak{S}_{r_m} \leq \mathfrak{S}_n=W_n
\end{align}
denote the subgroup preserving the decomposition $A=\coprod_{i=1}^mA_i$ (i.e., those permutations $\sigma$ such that $\sigma(V_{\lambda_i})=V_{\lambda_i}$ for all $i$).  We let 
$$
\Sigma_{\underline{\lambda}}':=\{\tau \in W_n:\textrm{ there exists }\xi \in \mathfrak{S}_m \textrm{ such that }\tau(A_i)=A_{\xi(i)} \textrm{ for all }1 \leq i \leq m\}.
$$
The map $\tau \mapsto \xi$ implicit in this definition is in fact a homomorphism $\Sigma_{\underline{\lambda}}' \to \mathfrak{S}_m$ whose image is the set of permutations fixing the multiplicity function $i \mapsto d_i$ and whose kernel is $\Sigma_{\underline{\lambda}}$.  
The Weyl group $W$ acts on $X^*(\widehat{T})$ and this action preserves the weights $\lambda_1,\dots,\lambda_m$ and the multiplicity function 
$i \mapsto d_i$. Thus the map $W \to \mathfrak{S}_m$ induces a morphism
$$
\rho_W:W \lto \Sigma_{\underline{\lambda}}'/\Sigma_{\underline{\lambda}}
$$
(this $\rho_W$ has nothing to do with the sum of positive roots).
We define $W'$ to be the following extension of $W$:
\begin{align}
W':=\{(w,\xi) \in W \times \Sigma_{\underline{\lambda}}': \rho_W(w) \equiv \xi \pmod{\Sigma_{\underline{\lambda}}}\}.
\end{align}
The group $W'$ admits natural homomorphisms to both $\Sigma_{\underline{\lambda}}'$ and $W$ by projection to the two factors.  We therefore obtain actions of $W'$ via projection to $\Sigma_{\underline{\lambda}}'$ on $\Lambda,$ $T_n$ and $\widehat{T}_n$ and actions of $W'$ via projection to $W$ on $X^*(\widehat{T})$, $T$ and $\widehat{T}$.  One checks that \eqref{basis-map} is $W'$ equivariant with respect to these actions, and hence so are the maps \eqref{pin} and \eqref{pin:dual}.  

For unramified quasicharacters $\chi:T(F) \to \CC^\times$ extend $\chi$ to a character of a Borel subgroup $B$ containing $T$.  Let
 $J(\chi)$ be the unique irreducible spherical subquotient of the unitarily normalized induction $\mathrm{Ind}_{B(F)}^{G(F)}(\chi)$ (see Theorem \ref{thm:sph} for more details and references).

Suppose that $\pi=J(\chi)$ where $\chi:T(F) \to \CC^\times$ is an unramified quasicharacter.  One defines $r(\chi)$ to be $\chi \circ r^{\vee}$ and
one defines the transfer of $\pi$ to be
$$
r(\pi):=J(r(\chi)).
$$
  This is an irreducible admissible representation of $\GL_n(F)$.  We note that if 
$$
r(\chi)\left(\begin{smallmatrix} a_1 & & \\ & \ddots & \\ & & a_n\end{smallmatrix}\right)=\prod_{i=1}^n\eta_i(a_i)
$$
for some quasicharacters $\eta_i:F^\times \to \CC^\times$ then
$$
\gamma(s,\pi,r,\psi):=\gamma(s,r(\pi),\psi)=\prod_{i=1}^n\gamma(s,\eta_i,\psi)
$$ 
\cite[Corollary 3.6 and Corollary 8.9]{GodementJacquetBook}.

\section{Preliminaries on the Fourier transform}

In this section we collect some notation related to Langlands decompositions of Borel subgroups and recall some basic facts about the spherical Fourier transform.  All of these results will be used in \S \ref{sec-main-thm} below.
In this section $F$ is an archimedian local field.

\subsection{Langlands decompositions}
\label{ssec-Langlands-decomp}
Let $T \leq G$ and $T_n \leq \GL_n$ be maximal tori as in \S \ref{sec-transfer} (so $T_n$ is the diagonal torus, $T$ is split and $r$ maps $\widehat{T}$ into $\widehat{T}_n$).  Let $B\geq T$ be a Borel subgroup of $G$ and let $B_n \geq T_n$ be the Borel subgroup of upper triangular matrices in $\GL_n$.  We let $N \leq B$ and $N_n \leq B_n$ be the unipotent radicals.  
Let $K \leq G(F)$ be a maximal compact subgroup and let $K_n \leq \GL_n(F)$ be the standard maximal compact subgroup.  
Let 
\begin{align*}
M:=T(F) \cap K, \quad
M_n:=T_n(F) \cap K_n\,.
\end{align*} 
We can and do assume that $M$ is the maximal compact subgroup of $T(F)$, and then one has $r^{\vee}(M_n) \leq M$.  
Let $\mathfrak{a}=X_*(T) \otimes_{\ZZ} \RR,\mathfrak{a}_n:=X_*(T_n) \otimes_{\ZZ}\RR$ and let 
\begin{align*}
 \mathfrak{a}^*&=\mathrm{Hom}(T(F)/M,\RR^\times_{>0}),\\
\mathfrak{a}_n^*&=\mathrm{Hom}(T_n(F)/M_n,\RR^{\times}_{>0})
\end{align*}
be their $\RR$-linear duals.
 
 We require a norm on $\mathfrak{a}^*_\CC$.  To construct it, let $(\, ,\,)$ be a nondegenerate symmetric bilinear form on $\mathfrak{g}:=\mathrm{Lie} (\mathrm{Res}_{F/\RR}G)$ whose restriction to the derived algebra is the Killing form.  We assume that the $+1$ and $-1$ eigenspaces of the Cartan involution $\Theta$ attached to $K$ are orthogonal under $(\, ,\,)$ and that $X \mapsto -(X,\Theta X)$ is a positive definite quadratic form on $\mathfrak{g}$.  We then set $||X||^2=-( X,\Theta X )$.  It induces a hermitian inner product on $\mathfrak{g}_\CC$ and $\mathfrak{a}_\CC$ and we continue to denote by $||\cdot||$ the induced form on $\mathfrak{a}_\CC^*$.

  The map $r^\vee$ yields
$$
r:\mathfrak{a}^*_\CC \lto \mathfrak{a}_{n\CC}^*\,.
$$
One has an isomorphism
\begin{align*}
\RR^n &\tilde{\lto} \mathfrak{a}_n^*\\
(s_1,\dots,s_n) &\longmapsto \eta_{s_1,\dots,s_n}
\end{align*}
where
$$
\eta_{s_1,\dots,s_n}\left( \begin{smallmatrix} t_1 & & \\ & \ddots & \\ & & t_n\end{smallmatrix}\right):= |t_1|^{s_1}\dots|t_n|^{s_n}.
$$
We let
$$
(\mathfrak{a}_{n}^*)_+
$$
be the image of $\RR_{>0}^n$ and let
\begin{align} \label{aplus}
\mathfrak{a}^*_+:=\{ \lambda \in \mathfrak{a}^*: r(\lambda) \in (\mathfrak{a}_n^*)_+\}.
\end{align}
By the existence of $\omega$ (see \eqref{omega}) this is nonempty.  
Note that this is not a Weyl chamber.

\subsection{Spherical functions}

 We now recall the definition of the Harish-Chandra map.  We define a function $H_T:T(F)/M \lto \mathrm{Hom}_{\ZZ}(X^*(T),\RR)$ via
$$
\langle \chi,H_T(x) \rangle:=\log |\chi(x)|. 
$$
Since there is a canonical identification $\mathrm{Hom}_{\ZZ}(X^*(T),\RR)=X_*(T) \otimes_{\ZZ}\RR=:\mathfrak{a}$ we can regard $H_T$ as taking values in $\mathfrak{a}$.
For 
$(k,t,n) \in K \times T(F) \times N(F)$ the Harish-Chandra map is then defined to be
\begin{align}
H_B:G(F) &\lto \mathfrak{a}\\
ktn &\longmapsto H_T(t).\nonumber
\end{align}
We choose Haar measures $dk$, $dt$, $dn$, $dg$
 on $K$, $T(F)$, $N(F)$, and $G(F)$, respectively, such that $\mathrm{meas}_{dk}(K)=1$ and for $f \in C_c^\infty(G(F))$
$$
\int_{G(F)}f(g)dg=\int_{K \times T(F) \times N(F)}e^{\langle 2\rho,H_B(t)\rangle} f(ktn)dkdtdn
$$
where $\rho \in \mathfrak{a}^*$ is half the sum of the positive roots of $T$ in $B$.  

For $\lambda \in \mathfrak{a}^*_\CC$ we let
\begin{align} \label{spherical:func}
\varphi_\lambda(g)=\int_K e^{\langle-(\lambda+\rho), H_B(g^{-1}k)\rangle}dk
\end{align}
be the usual spherical function.  It is a matrix coefficient of the representation $\mathrm{Ind}(e^{\langle\lambda,H_B\rangle})$.  One defines the spherical transform of suitable $K$-biinvariant continuous functions $f:G(F) \to \CC$ to be
\begin{align}
\widetilde{f}(\lambda):=\int_{G(F)} f(g)\varphi_{-\lambda}(g) dg
\end{align}
(here we are using the convention of \cite[\S 1]{AnkerSimple}, at least up to multiplication by $\sqrt{-1}$).

\subsection{Spaces of functions} \label{sec-sp-funcs}
Let  $\mathfrak{g}:=\mathrm{Lie}(\mathrm{Res}_{F/\RR}G)$. 
For $0<p \leq 2$ let
\begin{align} \label{C1-space}
\mathcal{S}^p(G(F)//K)
\end{align}
denote the space of $K$-biinvariant functions $f:G(F) \to \CC$ such that
\begin{align*}
 \mathrm{sup}_{x \in G(F)} (|x|+1)^n \varphi_0(x)^{-2/p}|X*f*Y(x)|<\infty
\end{align*}
for all $n \geq 0$ and all invariant differential operators $X$, $Y$ on $G(F)$, that is, all elements of the universal enveloping algebra $U(\mathfrak{g}_\CC)$ of the complexification of $\mathfrak{g}$.
Here $\varphi_0$ is the spherical function of \eqref{spherical:func} and $|x|$ is the distance of $x$ to $K$ (see, e.g.~\cite[\S 1]{AnkerSimple}).

It is known that for $0<p<p'\leq 2$ there are continuous inclusions 
$$
\mathcal{S}^p(G(F)//K) \leq \mathcal{S}^{p'}(G(F)//K) \leq L^{p'}(G(F))
$$ 
(see \cite[\S 1]{AnkerSimple}).  To check the continuity of the last inclusion one uses the fact that  $\frac{\varphi_0^{2/p}(x)}{(|x|+1)^n} \in L^p(G(F))$ 
for $n$ sufficiently large \cite[Proposition 4.6.12]{GangolliVaradarajan}.  

\subsection{The Fourier transform}
For $0<p \leq 2$ let $\mathfrak{a}_p^*$ be the closed tube in $\mathfrak{a}_\CC^*$ of points whose real part lies in the closed convex
hull of $W\cdot(\tfrac{2}{p}-1)\rho$ in $\mathfrak{a}^*$.  Here $\rho$ denotes half the sum of the positive roots of $T$ in $B$. Let
$$
\mathcal{S}(\mathfrak{a}_{p}^*)
$$
denote the space of all complex valued functions $h:\mathfrak{a}_p^* \to \CC$ such that 
\begin{enumerate}
\item[(a)] all derivatives of the function $h$ exist and are continuous on $\mathfrak{a}_p^*$,
\item[(b)] the function $h$ is holomorphic in the interior of $\mathfrak{a}_p^*$, 
\item[(c)] for any polynomial $P$ in the symmetric algebra of $\mathfrak{a}^*$ and any integer $n \geq 0$
$$
\mathrm{sup}_{\lambda \in  \mathfrak{a}^*_p}(\parallel\lambda\parallel+1)^n\left|P\left(\frac{\partial}{\partial \lambda} \right)h(\lambda) \right|<\infty.
$$
\end{enumerate}

Trombi and Varadarajan \cite{TrombiVaradarajanSpherical} proved  that the spherical Fourier transform $f \mapsto \widetilde{f}(\lambda)$ extends to an isomorphism of Frechet algebras
\begin{align} \label{FT2}
\mathcal{S}^p(G(F)//K) \tilde{\lto} \mathcal{S}(\mathfrak{a}_p^*)^W.
\end{align}
The seminorms which make $\mathcal{S}^p(G(F)//K)$ and $\mathcal{S}(\mathfrak{a}_p^*)$ into Frechet spaces are the obvious ones.
The paper \cite{AnkerSimple}, which contains a simpler proof of the isomorphism \eqref{FT2}, is a very nice reference for the facts above, though, strictly speaking, it assumes that $G$ is semisimple.  In loc.~cit.~$\mathfrak{a}_p^*$ is denoted $i\mathfrak{a}_{2/p-1}^*$.  

For $0 <p \leq 1$, $f \in \mathcal{S}^p(G(F)//K)$ let
$$
f^{(B)}(t)=e^{\langle \rho, H_B(t)\rangle}\int_{N(F)} f(tn)dn
$$
be the constant term of $f$ along $B$.  It is absolutely convergent \cite[Theorem 6.2.4]{GangolliVaradarajan}.

Let $\chi:T(F) \to \CC^\times$ be a character.  If $\chi=e^{\langle \lambda , H_B \rangle}$ for some $\lambda \in \mathfrak{a}^*_\CC$ whose real part is  in the closed Weyl chamber defined by the positive roots attached to $B$
we will abbreviate 
$$
J(\lambda):=J(e^{\langle\lambda, H_B\rangle})\quad\text{and}\quad \mathrm{Ind}(\lambda):=\mathrm{Ind}_{B(F)}^{G(F)}(e^{\langle \lambda, H_B\rangle}).
$$
We recall the following special case of the Langlands classification:

\begin{thm} \label{thm:sph} Any irreducible admissible spherical representation of $G(F)$ is infinitisimally equivalent to $J(\lambda)$ for some $\lambda \in \mathfrak{a}_\CC^*$ whose real part is in the closed positive Weyl chamber attached to $B$.   
Conversely, for every $\lambda \in \mathfrak{a}_\CC^*$ whose real part is in the closed positive Weyl chamber with respect to $B$ 
the representation $\mathrm{Ind}(\lambda)$ has a unique irreducible quotient (usually called the Langlands quotient).  It is spherical.
\end{thm}

\begin{proof}
In the notation of \cite{Vogan:Rep:real} and \cite{BarbaschCiubotaruPantano} take $\delta=\mathrm{triv}$ and $\mu=\mathrm{triv}$.  In the terminology of loc.~cit.~$\delta$ is a fine $T(F) \cap K$-type, $\mu$ is a fine $K$-type and $\mu \in A(\delta)$.  The stabilizer $W_\delta$ of $\delta$ under the natural action of the Weyl group $W$ is all of $W$: $W_\delta=W$.  Thus the first assertion of the lemma follows from \cite[Theorem, \S 2.4]{BarbaschCiubotaruPantano}.  

On the other hand, since $\delta$ is trivial the $R$-group $R_\delta$ is trivial (see \cite[Definition 4.3.13]{Vogan:Rep:real}), hence the set $A(\delta)$ consists of one element, namely the trivial $K$-type (see \cite[Theorem 4.3.16]{Vogan:Rep:real}).  In view of this, the final assertion of the theorem is contained in \cite[ \S 2.11]{BarbaschCiubotaruPantano}.
\end{proof}

We also give a proof of the following well-known theorem for the convenience of the reader.

\begin{lem} \label{lem-FT} Let $\lambda \in \mathfrak{a}_\CC^*$ have real part in the closed positive Weyl chamber with respect to $B$.
Suppose that $J(\lambda)$ is unitary.  
If $f \in C_c^\infty(G(F)//K)$ then $J(\lambda)(f)$ acts via the scalar
\begin{align*}
\widetilde{f}(-\lambda)=\mathrm{tr}\,J(\lambda)(f)=\mathrm{tr}\,\mathrm{Ind}(\lambda)(f)=\mathrm{tr}\,e^{\langle\lambda, H_B\rangle}(f^{(B)})
\end{align*}
on the (unique) spherical vector in $J(\lambda)$. Under the same assumptions on $J(\lambda)$, if $0<p \leq 1$ and $f \in \mathcal{S}^p(G(F)//K)$ then $J(\lambda)(f)$ acts via the scalar
$$
\widetilde{f}(-\lambda)=\mathrm{tr}\,J(\lambda)(f)
$$
on the spherical vector in $J(\lambda)$.
\end{lem}

\begin{proof}
Assume first that $f \in C_c^\infty(G(F)//K)$.
The identity $\mathrm{tr}\,\mathrm{Ind}(\lambda)(f)=\mathrm{tr}\,e^{\langle \lambda ,H_B\rangle}(f^{(B)})$ is the descent formula (see, e.g.~\cite[(10.23)]{KnappSS}).  The vector $\varphi_{\lambda}$ is spherical, satisfies $\varphi_\lambda(1)=1$, and is a matrix coefficient of the representation $\mathrm{Ind}(\lambda)$ (compare \cite[\S 1]{AnkerSimple}).  It is also known that this representation, even if it is reducible, contains a unique spherical line \cite[3.1.13]{GangolliVaradarajan}.  It follows that $\mathrm{Ind}(\lambda)(f)$ acts via the scalar $\widetilde{f}(-\lambda)=\mathrm{tr}\,\mathrm{Ind}(\lambda)(f)$ on this spherical line.  On the other hand, one has a nonzero equivariant map
\begin{align} \label{equiv}
\mathrm{Ind}(\lambda) \lto J(\lambda ).
\end{align}
Since the irreducible representation $J(\lambda)$, being spherical, has a unique spherical line this line must be the image of the unique spherical line in $\mathrm{Ind}(\lambda)$ under \eqref{equiv}.  Thus $\mathrm{tr}\,J(\lambda)(f)=\mathrm{tr}\,\mathrm{Ind}(\lambda)(f)$.

In the following discussion we use basic facts that are recalled in \cite[\S 1]{AnkerSimple} without further comment.  Let $0<p \leq 1$ and $f \in \mathcal{S}^p(G(F)//K)$.  
Then $\widetilde{f}$ is defined on $\lambda \in \mathfrak{a}_1^*$.  Since $C_c^\infty(G(F)//K)$ is dense in $\mathcal{S}^p(G(F)//K)$ we can choose a Cauchy sequence $\{f_n\}_{n=1}^\infty \subset C_c^\infty(G(F)//K)$ converging to $f$ in $\mathcal{S}^p(G(F)//K)$.  In particular,
$\lim_{n \to \infty} \widetilde{f}_n=\widetilde{f}$ pointwise on $\mathfrak{a}_1^*$.  

Now the inclusion $\mathcal{S}^p(G(F)//K) \to L^1(G(F)//K)$ is continuous. Since $J(\lambda)$ is unitary the fact that $\lim_{n \to \infty}f_n = f$ in $L^1(G(F))$ implies that 
$$
\lim_{n\to \infty}\mathrm{tr}\,J(\lambda)(f_n)=\mathrm{tr}\,J(\lambda)(f).
$$
  Thus 
$$
\widetilde{f}(-\lambda)=\lim_{n \to \infty} \widetilde{f}_n(-\lambda)=\lim_{n \to \infty} \mathrm{tr}\,J(\lambda)(f_n)= \mathrm{tr}\,J(\lambda)(f).
$$
\end{proof}

\section{Proof of Theorem \ref{thm-lfe}} \label{sec-main-thm}

Recall that we have normalized $r:\widehat{G} \to \GL_n$ so that it induces a morphism $r:\widehat{T} \to \widehat{T}_n$ and used it to construct a dual morphism
$$
r^{\vee}:T_n \lto T.
$$
There are natural isomorphisms 
$$
T_n(F)/M_n \cong X_*(T_n) \otimes_{\ZZ}\RR=:\mathfrak{a}_n,
$$
 and $T(F)/M \cong \mathfrak{a}$, so $r^\vee$ induces an $\RR$-linear map
$$
r^{\vee}:\mathfrak{a}_n:=T_n(F)/M_n \lto T(F)/M=:\mathfrak{a}.
$$
It is surjective, as the complexification of its dual is
$$
r:\mathfrak{a}_\CC^* \lto \mathfrak{a}_{n\CC}^*
$$
which is injective since we assumed $r$ has zero dimensional kernel.  Recall the group $W'$ of \S \ref{sec-transfer}.  It acts naturally on 
$\mathfrak{a}$, $\mathfrak{a}_n$, $\mathfrak{a}_{\CC}^*$, $\mathfrak{a}_{n\CC}^*$ and the maps $r$ and $r^{\vee}$ are $W'$ equivariant (see \S \ref{sec-transfer}).  We therefore obtain a push-forward map
\begin{align} \label{rpush}
r_*^{\vee}:C_c^\infty(T_n(F)/M_n)^{W'} \lto C_c^\infty(T(F)/M)^{W'}=C_c^\infty(T(F)/M)^W.
\end{align}
We define
\begin{align} \label{r:space}
C_c^\infty(G(F)//K,r):=\left\{f \in C_c^\infty(G(F)//K): f^{(B)} \in r_*^{\vee}(C_c^\infty(T_n(F)/M_n)^{W'})\right\}.
\end{align}

Let $d_{\omega} \in \mathfrak{a}^*$ be the point corresponding to $\omega^\vee$.  
In this section we prove the following theorem.  It obviously implies Theorem \ref{thm-lfe}; it is simply a version of Theorem \ref{thm-lfe} that makes explicit how large $\mathrm{Re}(s)$ must be in terms of the representation $r$.
\begin{thm}\label{thm-lfe2} 
Let $f \in C_c^\infty(G(F)//K,r)$ and  $0<p \leq 1$.  There is 
a function
$$
\mathcal{F}_{r,\psi}(f) \in C^\infty(G(F)//K)
$$
such that 
\begin{enumerate}
\item[(a)] one has
$
\mathcal{F}_{r,\psi}(f)\omega_s \in \mathcal{S}^p(G(F)//K)
$
provided that 
$
\mathrm{Re}(\mathfrak{a}_p^*)
+\mathrm{Re}(s)d_\omega \subset \mathfrak{a}_+^*
$  
and

\item[(b)] if $\pi$ is unitary and irreducible then 
\begin{align} \label{loc-fe-1}
\gamma(s,\pi,r,\psi)\mathrm{tr}\,\pi_{s}(f)=\mathrm{tr}\,\pi^\vee_{1-s}(\mathcal{F}_{r,\psi}(f))
\end{align}
in the sense of analytic continuation.
\end{enumerate}
\end{thm}

Now we give a diagram that outlines the construction of 
$\mathcal{F}_{r,\psi}(f)$ from $f$.  Let $s_0>0$ be large enough that 
$\mathrm{Re}(\mathfrak{a}_p^*)+s_0d_{\omega} \subset \mathfrak{a}_+^*$.  
Let
 $\mathcal{S}(\mathfrak{a}_p^*)(-s_0)$ (resp.~$\mathcal{S}(\mathfrak{a}_p^*)(-s_0)^W$) denote the space of functions of the form 
 \begin{align*}
 \mathfrak{a}_p^*+s_0d_\omega & \lto \CC\\
\lambda &\longmapsto  \widetilde{f}(\lambda-s_0d_\omega)
\end{align*} 
for some $\widetilde{f} \in \mathcal{S}(\mathfrak{a}_p^*)$ (resp.~$\mathcal{S}(\mathfrak{a}_p^*)^W$).  We note that $\omega$ is fixed by $W$, so this space admits an action of $W$, and $ \mathcal{S}(\mathfrak{a}_p^*)^W(-s_0)= \mathcal{S}(\mathfrak{a}_p^*)(-s_0)^W$.  Moreover, \eqref{FT2} induces an isomorphism
\begin{align*}
\mathcal{S}^p(G(F)//K)\omega_{-s_0} &\tilde{\lto} \mathcal{S}(\mathfrak{a}_p^*)(-s_0)^W\\
f &\tilde{\longmapsto} \widetilde{f}(\lambda)
\end{align*}
where the $\CC$-vector space on the left is the space 
of functions of the form $f\omega_{-s_0}$ for $f \in \mathcal{S}^p(G(F)//K)$.
We will construct $\mathcal{F}_{r,\psi}(f)$ so that the following diagram commutes:
\begin{center}
\begin{tikzcd} 
C_c^\infty(T_n(F)/M_n)^{W'} \arrow[r,"\widehat{}"] \arrow[d,"r^{\vee}_*"] &  \mathcal{S}(\mathfrak{t}_n(F)/M_n)^{W'} \arrow[d, "\Psi \mapsto \widetilde{\Psi} \circ r"] 
 \\ 
C_c^{\infty}(T(F)/M)^W
& \mathcal{S}(\mathfrak{a}_p^*)(-s_0)^W\\
C_c^\infty(G(F)//K) \arrow[u, "f \mapsto f^{(B)}"] \arrow[r, "\mathcal{F}_{r,\psi}"]
&\mathcal{S}^p(G(F)//K)\omega_{-s_0} \arrow[u,"f \mapsto \widetilde{f}(-\lambda)"]
\end{tikzcd}
\end{center}

\noindent Here the horizontal arrow marked \,$\widehat{\textrm{ }}$\, is \eqref{FT3}, which is just the Fourier transform on $T_n(F) \subset \mathfrak{t}_n(F)$ (the $F$-points of the Lie algebra of $T_n$).  Moreover $\widetilde{\Psi}$ is a Mellin transform.
The proof of Theorem \ref{thm-lfe2} we now give amounts to filling in the details of this diagram. The functional equation ultimately reduces to the familiar functional equation of Tate zeta functions.

Before beginning the proof in earnest let us describe the map $\Psi \mapsto \widetilde{\Psi} \circ r$ more precisely and show that it is well defined.  
Let $\mathfrak{t}_n:=\mathrm{Lie}(T_n)$ and denote by $\mathcal{S}(\mathfrak{t}_n(F))$ the usual Schwartz space.  Assume that 
 $$
 \Psi \in \mathcal{S}(\mathfrak{t}_n(F)).
 $$  
 We then have that
\begin{align} \label{Mellin}
\widetilde{\Psi}(\lambda):=\int_{T_n(F)}e^{\langle\lambda ,H_{T_n}(t)\rangle}\Psi(t)dt
\end{align}
is absolutely convergent and holomorphic if 
$\mathrm{Re}(\lambda) \in (\mathfrak{a}_{n}^*)_+$.  Here, as above, $dt$ is the Haar measure on $T_n(F)$.  In fact something stronger is true:
\begin{lem} \label{lem-bounds} For any polynomial $P$ in the symmetric algebra of $\mathfrak{a}_n^*$  and any integer $n \geq 0$ the quantity
$$
(\parallel\lambda \parallel+1)^n\left|P\left(\frac{\partial}{\partial \lambda} \right)\widetilde{\Psi}(\lambda) \right| 
$$
is bounded for $\mathrm{Re}(\lambda)$ in a fixed compact subset of 
$(\mathfrak{a}_{n}^*)_+$.  Here $||\cdot||$ is any Hermitian inner product on $\mathfrak{a}_{n\CC}^*$.
\end{lem}
\begin{proof}
The lemma is a consequence of the following claim: for any $f \in \mathcal{S}(\RR)$, real numbers $0<a<b$, and nonnegative integers $n,k$  one has that
$$
|s|^n\left|\frac{d^k}{d s^k}\widetilde{f}(s) \right| \ll_{f,a,b,n,k} 1
$$
provided that $a \leq \mathrm{Re}(s) \leq b$.  Here $\widetilde{f}(s)=\int_{0}^\infty f(x)x^{s-1}dx$ is the usual Mellin transform.  To check this, let $D:=x\frac{d}{dx}$.  It is not hard to see that $x^{\varepsilon} D^n\left(\log^k f \right)$ is continuous on $\RR_{>0}$ and in $L^1\left(\RR_{>0},\frac{dx}{x}\right)$ for any real number $\varepsilon>0$, and thus the Mellin transform $\left(D^n\left(\log^k f \right)\right)^{\sim}(\sigma+it)$ is bounded as a function of $t \in \RR$ for $a \leq \sigma \leq b$.  Thus
$$
\left|s^n\frac{d^k}{d s^k}\widetilde{f}(s)\right|= \left|\left(D^n\left(\log^kf\right) \right)^{\sim}(s)\right|
$$
is bounded for $a \leq \mathrm{Re}(s) \leq b $.  
\end{proof}

\begin{cor}  \label{cor:Mellin}
Suppose that $s_0 \in \RR_{>0}$ is large enough that 
$
\mathfrak{a}_p^*+s_0d_\omega \subset \mathfrak{a}^*_+$ and that $\Psi \in \mathcal{S}(\mathfrak{t}_n(F)/M_n)$.  Then the function 
$$
\lambda \longmapsto \widetilde{\Psi}(r(\lambda))
$$ 
is in $\mathcal{S}(\mathfrak{a}_p^*)(-s_0)$.
\end{cor}

\begin{proof}
By assumption and the definition \eqref{aplus} of $\mathfrak{a}^+$ one has $r(\mathfrak{a}_p^*+s_0d_{\omega}) \subset (\mathfrak{a}_n^*)_+$.  Thus the corollary follows from Lemma \ref{lem-bounds} and the injectivity of the map $r:\mathfrak{a}^*_\CC \to \mathfrak{a}^*_{n\CC}$.
\end{proof}

\noindent The corollary implies that the map $\Psi \mapsto \widetilde{\Psi} \circ r$ in the top right of the diagram is well-defined for $s_0$ large enough.

We now begin the proof of Theorem \ref{thm-lfe2}.   Let $\pi$ be a given spherical unitary irreducible representation of $G(F)$.  Thus there is a quasi-character
$\chi:T(F)/M \to \CC^\times$ so that $\pi \cong J(\chi)$.  Let $f \in C_c^\infty(G(F)//K,r)$.  We will trace $J(\chi)$ and $f$ along the upper path from $C_c^{\infty}(G(F)//K)$ to $\mathcal{S}^p(G(F)//K)\omega_{-s_0}$ in the diagram.
 Lemma \ref{lem-FT} implies that
\begin{align} \label{des-rel}
\mathrm{tr}\,J(\chi)(f)=\chi(f^{(B)}).
\end{align}

We choose a $\Phi \in C_c^\infty(T_n(F)/M_n)^{W'}$ so that the push-forward
$
r^{\vee}_*(\Phi)$ is $f^{(B)}$.  
 Thus
\begin{align} \label{first-eq}
\mathrm{tr}\,J(\chi)_s(f)=\chi \omega_s(f^{(B)})=\chi\omega_s(r^{\vee}_*(\Phi))=r(\chi)|\det|^s(\Phi).
\end{align}

We are now at the top row of the diagram.
The usual embedding $\GL_n \hookrightarrow \gl_n$ of algebraic monoids induces an embedding $T_n \hookrightarrow \mathfrak{t}_n$ where $\mathfrak{t}_n:=\mathrm{Lie}(T_n)$.  
We can therefore regard an element of $C_c^\infty(T_n(F))$ as an element of $C_c^\infty(\mathfrak{t}_n(F))$.  The pairing
\begin{align*}
\mathfrak{t}_n(F) \times \mathfrak{t}_n(F) &\lto \CC\\
(X,Y)& \longmapsto \psi(\mathrm{tr}(XY))
\end{align*}
is perfect.  For $t \in T_n(F)$ let
\begin{align} \label{FT3}
\widehat{\Phi}(t)=\int_{\mathfrak{t}_n(F)}\Phi(x)\psi(\mathrm{tr}(tx)) dx;
\end{align}
it is just the Fourier transform.  We then have
\begin{align} \label{fotr}
\gamma(s,J(r(\chi)),\psi)r(\chi)|\det|^s(\Phi)=r(\chi^{-1})|\det|^{1-s}(\widehat{\Phi})
\end{align}
by the local functional equation of Tate's thesis \cite[(3.2.1)]{Tate_NT}.  Here the left hand side is meromorphic as a function of $s$, and $\mathrm{tr}\,r(\chi^{-1})|\det|^{1-s}(\widehat{\Phi})$ is absolutely convergent for $\mathrm{Re}(s)$ sufficiently small in a sense depending on $\chi$.

We are now at the upper right corner of the diagram.  
The representation $J(\chi)$ is unitary, but $\chi$ need not be.  
However, there is a $\mu \in \mathfrak{a}_\CC^*$ whose real part lies in the the closed convex hull of $W.\rho$ in $\mathfrak{a}^*$ such that $\chi(t)=e^{\langle \mu ,H_T(t)\rangle}$ by \cite[p. 654]{KnappSS}.  In particular, since $0<p \leq 1$ one has $\mu \in \mathfrak{a}_p^*$.  Then 
\begin{align} \label{mu:def}
r(\chi^{-1})=e^{\langle r(-\mu) , H_T\rangle}.
\end{align}
Combining this notation with \eqref{first-eq} and \eqref{fotr} we arrive at
\begin{align} \label{ft-2}
\gamma(s,J(r(\chi)),\psi)\mathrm{tr}\,J(\chi)_s(f)&=\gamma(s,J(r(\chi)),\psi)r(\chi)|\det|^s(\Phi)\\
&=r(\chi^{-1})|\det|^{1-s}(\widehat{\Phi}) \nonumber\\
&=\widehat{\Phi}^{\sim}(r(-\mu+(1-s)d_\omega )). \nonumber
\end{align}
Now $\lambda \mapsto \widehat{\Phi}^{\sim}(r(\lambda))$ is in $\mathcal{S}(\mathfrak{a}_p^*)(-s_0)$ for $s_0$ sufficiently large by Corollary \ref{cor:Mellin}.  Thus there is a unique $h \in \mathcal{S}^p(G(F)//K)$ so that 
 for all $\lambda \in \mathfrak{a}_p^*$ one has
 \begin{align*}
 \widehat{h}(-\lambda)=\widehat{\Phi}^{\sim}(r(\lambda+s_0d_\omega))
 \end{align*}
 by \eqref{FT2}.
 In particular, if $\lambda$ is chosen so that $J(\lambda)$ is spherical and unitary then 
\begin{align} \label{2-eq}
\mathrm{tr}J(\lambda)(h)=
\widehat{\Phi}^{\sim}(r(\lambda+s_0d_\omega))
\end{align}
by  Lemma \ref{lem-FT}.  
Set
$$
\mathcal{F}_{r,\psi}(f):=h\omega_{-s_0}.
$$
By construction, $h \in \mathcal{S}^p(G(F)//K)$.  We have now successfully traversed along the upper path from $C_c^\infty(G(F)//K)$ to $\mathcal{S}^p(G(F)//K)\omega_{-s_0}$ in our diagram.

Combining \eqref{ft-2} and \eqref{2-eq} we deduce that (with $\mu$ as in \eqref{mu:def}),
\begin{align} \label{loc-fe}
\gamma(s,J(\chi),r,\psi)\mathrm{tr}\,J(\chi)_{s}(f)
&=\widehat{\Phi}^{\sim}(r(-\mu+(1-s)d_\omega))\\&=\mathcal{F}_{r,\psi}(f)^\sim(-\mu+(1-s)d_\omega) \nonumber \\&
=\mathrm{tr}\,J(\chi)^\vee_{1-s}(\mathcal{F}_{r,\psi}(f)). \nonumber
\end{align}
This completes the proof of the theorem.  
\qed

\section{A global application} \label{sec:app}

Let $F$ be a number field and let $\infty$ be the set of infinite places of $F$.  Let $G=\GL_m$ for some integer $m$.  It is not strictly necessary to take $G=\GL_m$ for what follows, but it makes the discussion simpler.  We also restrict ourselves to everywhere unramified representations.  We assume $\omega=\det:G \to \GG_m$.

\begin{rem} We have already discussed how one might remove the unramified assumption at the archimedean places in the introduction.
To treat representations that are ramified at finite places one would have to define nonarchimedean Fourier transforms and give some analytic control on them similar to the control afforded in the archimedean setting by Theorem \ref{thm-lfe}.  More specifically, one would need to show that a twist of them by $\omega_{s}$ is $L^1$ for $\mathrm{Re}(s)$ sufficiently large.  For $G=\GL_m$ one can define these nonarchimedean Fourier transforms spectrally using the Plancherel formula since the local Langlands correspondence is known.   This is the approach of \cite{LafforgueJJM}; in loc.~cit.~the analytic control is not established.  Cheng and Ng\^o's approach in \cite{Cheng:Ngo:BK}, if it is generalized from the finite field case to the local field case, may also yield the desired Fourier transforms.  
\end{rem}

 We retain the obvious analogues of the notation above in this global setting; for example $T \leq \GL_m$ is a maximal split torus which we take to be the diagonal matrices for simplicity.
For $v \nmid \infty$ let
\begin{align} \label{Sat}
\mathcal{S}:C_c^\infty(G(F_v)//G(\OO_{F_v})) \lto \CC[\widehat{T}]^{W}
\end{align}
be the Satake isomorphism.
Let 
$$
r:\widehat{G} \lto \GL_m
$$ 
be an irreducible representation.  Let $\A_F^\infty$ denote the ring of finite adeles of $F$ and let
$$
\LL_r:=\prod_{v \nmid \infty}\LL_{r,v} \in C_{ac}^\infty(G(\A^\infty_F)//G(\widehat{\OO}_F))
$$
where
$$
 \LL_{r,v}:=\sum_{k=0}^\infty \mathcal{S}^{-1}\left(\mathrm{tr}\,\mathrm{Sym}^k\left( r(t) \right)\right) \in C^\infty_{ac}(G(F_v)//G(\OO_{F_v}))
$$
with $t \in \widehat{T}(\CC)$.  Here the subscript $ac$ denotes the space of functions that are almost compactly supported, in other words, when restricted to a subset of $G(\A_F^\infty)$  with determinant lying in a compact subset of $(\A_F^\infty)^\times$ they are compactly supported (and similarly in the local setting).  Then if $\pi^\infty$ is an irreducible  unramified admissible representation of $G(\A_F^\infty)$ one has
\begin{align} \label{unram:Lfunc}
\mathrm{tr}\,\pi_{s}^\infty(\LL_{r})=L(s,\pi^\infty,r)
\end{align}
for $\mathrm{Re}(s)$ large enough.  Here $\pi_s^\infty:=\pi^\infty(|\omega|^{\infty})^{s/N}$.

Let $A$ be the connected component of the real points of the maximal $\QQ$-split torus in the center of $\mathrm{Res}_{F/\QQ}G$ and 
$$
G(\A_F)^1:=\mathrm{ker}\left(|\cdot|_{\A_F} \circ \det:G(\A_F) \lto \RR_{>0} \right).
$$

In \cite{LanglandsBE} Langlands proposed a method to prove Langlands functoriality in general via the trace formula.  His point of departure was that for $\bar{f} \in C_c^\infty(A \backslash G(\A_{F\infty}))$ one can use the trace formula to provide an absolutely convergent expression for
\begin{align}
\sum_{\pi} \mathrm{tr}\,\pi_\infty(\bar{f})\pi^\infty_s(\LL_{r})
\end{align}
for $\mathrm{Re}(s)$ large.  Here the sum over $\pi$ is over cuspidal automorphic representations of $A \backslash G(\A_F)$.  We note that the sum here is infinite, but since $\mathrm{tr}\,\pi^{\infty}_s(\LL_r)$ is bounded independently of $s$ and $\pi$ for $\mathrm{Re}(s)$ sufficiently large (compare the proof of Corollary \ref{cor:L1} below) and the regular action of any element of $C_c^\infty(A \backslash G(\A_F))$ on the
cuspidal spectrum is trace class the sum is absolutely convergent.
Strictly speaking, Langlands used logarithmic derivatives of $L$-functions. Sarnak proposed the current formulation because it appears to be  more tractable analytically \cite{SarnakLetter}.  

One can then hope to use the trace formula to give an expression for
\begin{align} \label{res:sum}
\sum_\pi \mathrm{Res}_{s=1} \mathrm{tr}\, \pi_\infty(\bar{f})L(s,\pi^\infty,r)
\end{align}
in terms of orbital integrals (and automorphic representations on Levi subgroups).
The residues ought to be nonzero for representations whose $L$-parameter, upon composition with $r$, fixes a vector in $V$.  These ought to be transfers from smaller groups, and one hopes to compare the sum of residues
\eqref{res:sum} with corresponding sums on smaller groups.  Since every algebraic subgroup of $\widehat{G}$ is the fixed points of a line in some representation of $\widehat{G}$ by a theorem of Chevalley, in principle executing this approach would lead to a proof of functoriality in general.

However, Langlands gave no absolutely convergent geometric expression for \eqref{res:sum} nor any indication of how to obtain one, even assuming Langlands functoriality.  In practice this seems to be an extremely difficult analytic hurdle that has only been overcome in a handful of cases \cite{AltugBEI,Getz4var,HermanRS,GetzHerman,VenkateshThesis,WhiteBE} that are essentially those isolated as tractable by Sarnak in his letter \cite{SarnakLetter}.  

In this section we use Theorem \ref{thm-lfe} and work of Finis, Lapid, and M\"uller to give an absolutely convergent expression in terms of orbital integrals and automorphic representations of Levi subgroups that is 
equal to \eqref{res:sum} if one assumes Langlands functoriality (what we need is given precisely in Conjecture \ref{conj:Langl} below).  We emphasize that the expression makes sense without any assumption in place, so one could try to use it to study Langlands' beyond endoscopy proposal.  At very least it allows us to replace \eqref{res:sum} with a quantity which is well-defined without any assumptions.

\begin{rem} For some time Ng\^o has advocated combining Braverman and Kazhdan's proposal in \cite{BK-lifting} with the trace formula to prove functional equations of $L$-functions.  The author first learned of this idea from Ng\^o at IAS in 2010, and Ng\^o has given a progress report on his perspective at the 2016 Takagi lectures. The construction we now propose, which is due to the author, amounts to understanding the residues that occur when one follows Ng\^o's suggestion.  In particular one can give an absolutely convergent expression for the sum of residues that is the focus of Langlands' beyond endoscopy proposal.
\end{rem}

For a compact open subgroup $K \leq G(\A_F^\infty)$ let
$$
\mathcal{C}(G(\A_F),K)=\{f:G(\A_F)/K \to \CC: |f*X|_{L^1(G(\A_F))} <\infty \textrm{ for all }X \in U(\mathfrak{g}_\CC)\}.
$$
Here $\mathfrak{g}$ is the Lie algebra of 
$$
G(\A_{F\infty})
$$ 
(viewed as a real Lie algebra) and $U(\mathfrak{g}_\CC)$ is its universal enveloping algebra.  This is the space of test functions treated in \cite{FinisLapidMullerAbsConv} and \cite{FinisLapid2011}.  A slightly different space of test functions (called $\mathcal{C}(G(\A_F)^1,K)$) is considered in \cite{FinisLapidContinuityGeometric}.  
  In any case, the main result of these papers is that Arthur's noninvariant trace formula is valid for functions in $\mathcal{C}(G(\A_F),K)$.

Let $K_\infty \leq G(\A_{F\infty})$ be a maximal compact subgroup and let 
$$
f \in \otimes_{v|\infty}C_c^\infty(G(F_v)//K_v,r)
$$
(see \eqref{r:space}).
  Let $\mathcal{F}_{r,\psi}(f):G(\A_{F\infty}) \to \CC$ be its non-abelian Fourier transform.  For $g \in G(\A_F)$ and $s \in \CC$ we set 
\begin{align*}
f\LL_r\omega_s(g):&=f\omega_{s \infty}(g_\infty)\LL_r\omega_{s}(g^\infty),\\
\mathcal{F}_{r,\psi}(f)\LL_r\omega_s(g):&=\mathcal{F}_{r,\psi}(f)\omega_{s\infty}(g_\infty)\LL_r\omega_s^\infty(g^\infty).
\end{align*}
Fix a nontrivial additive character $\psi:F \backslash \A_F \to \CC^\times$ and choose $d \in \A_F^{\times}$ such that $\psi^\infty(d \cdot)$ is unramified and $d_\infty \in A_{\GG_m},$ the copy of $\RR_{>0}$ embedded diagonally in $F_\infty^\times$.  We assume moreover that $\omega_s(d^NI_m)=|d^\infty|^{ns}$.
 Choose $\sigma \in \RR$ and let
\begin{align*}
f_1(g)&=f\LL_r\omega_{\sigma},\\
f_2(g)&=|d^{\infty}|^{-n}\mathcal{F}_{r,\psi}(f)\mathbb{L}_r\omega_{\sigma}(d^NI_m g).  
\end{align*}
  The role of the $d$ here is explained by Conjecture \ref{conj:Langl} below.

The key consequence of Theorem \ref{thm-lfe2} we use here is the following corollary:

\begin{cor} \label{cor:L1}  If $\sigma$ is sufficiently large then  $f_1,f_2 \in \mathcal{C}(G(\A_F),G(\widehat{\OO}_F))$. 
\end{cor}

\begin{proof} 
For $\sigma$ sufficiently large one has
\begin{align}
\int_{G(\A_F^\infty)}|\LL_r(g)|\omega_\sigma(g)dg<\infty
\end{align}
by \cite[Proposition 3.11]{WWLiSat}. Combining this with Theorem \ref{thm-lfe2} we immediately deduce the corollary.
\end{proof}

For a $G(F)$-conjugacy class $\mathfrak{o}$ in $G(F)$ and an $h \in \mathcal{C}(G(\A_F),K)$ Finis and Lapid \cite{FinisLapidContinuityGeometric} have shown that one can define the noninvariant orbital integral
$
J_{\mathfrak{o}}(h).
$
Technically speaking they work with a slightly coarser notion than conjugacy, but it reduces to conjugacy for the case at hand since we have assumed $G$ is a general linear group.

Together with M\"uller \cite{FinisLapidMullerAbsConv} they have also shown that for functions in the same space one can define a trace
\begin{align} \label{trace}
\int_{i\mathfrak{a}^*_{L_s}}\mathrm{tr}(\mathcal{M}_L(P,\lambda)M(P,s)\rho(P,\lambda,h)) d\lambda
\end{align}
that is again absolutely convergent.  Unfortunately, it would take several pages to define the notation used in \eqref{trace}, we refer the reader to \cite[Corollary 1]{FinisLapidMullerAbsConv} and the discussion preceding it.  This is the contribution of the Levi subgroup $L_s$ to the trace formula.

Call a parabolic subgroup of $G$ standard if it contains $T$.  For each standard parabolic subgroup $P$ let $M_P$ be the unique Levi subgroup of $P$ containing $T$.  For a cuspidal automorphic representation $\pi$ of $A \backslash G(\A_F)$ let
$$
\pi_s:=\pi |\omega|^{s/N}.
$$
The transfer $r(\pi)$ of $\pi$ to $\GL_n(\A_F)$ is an irreducible admissible representation (which of course, we do not yet know to be automorphic).  We let $\omega_{r(\pi)}$ be its central character.  

Using Corollary \ref{cor:L1} we can now apply the work of Finis, Lapid and M\"uller to prove the following theorem:

\begin{thm}
Consider
\begin{align} \label{spec:side2}
\sum_{\pi} \left(\frac{1}{2\pi i}\int_{\mathrm{Re}(s)=\sigma}\left( \mathrm{tr}\,\pi_{s}(f\LL_{r}) - \frac{|d^\infty|^{n(s-1)}}{\omega_{r(\pi)}(d^\infty)}\mathrm{tr}\,(\pi^\vee)_{s}(\mathcal{F}_{r,\psi}(f) \LL_r)\right)ds \right)
\end{align}
where the sums are over isomorphism classes of cuspidal automorphic representations of $A \backslash G(\A_F)$.  
It is equal to 
\begin{align}\label{to:show}
\sum_{\mathfrak{o}} J_{\mathfrak{o}}(f_1-f_2)
-\sum_{[P] \neq G} \frac{1}{|W(M_P)|} \sum_{s \in W(M_P)} \iota_s \int_{i\mathfrak{a}_{L_s}}\mathrm{tr}(\mathcal{M}_L(P,\lambda)M(P,s)\rho(P,\lambda,f_1-f_2)d\lambda
\end{align}
where the sum over $\mathfrak{o}$ is over conjugacy classes in $G(F)$, 
the sum over $[P]$ is over associate classes of standard proper parabolic subgroups of $G$, $W(M_P)$ is the Weyl group of $M_P$ in $G$, and $\iota_s$ is the normalizing factor of \cite[Corollary 1]{FinisLapidMullerAbsConv}.
Moreover for each $i$ and each $[P],s$
\begin{align*}
\sum_{\mathfrak{o}} |J_{\mathfrak{o}}(f_i)|<\infty
\textrm{ and }
\int_{i\mathfrak{a}_{L_s}}|\mathrm{tr}(\mathcal{M}_L(P,\lambda)M(P,s)\rho(P,\lambda,f_i)|d\lambda<\infty.
\end{align*}
\end{thm}

\begin{rem} The sum over $[P]$ and $s$ above is finite.  
\end{rem}

\begin{proof} Let $h \in \mathcal{C}(G(\A_F),G(\widehat{\OO}_F))$.
In \cite{FinisLapidMullerAbsConv} and \cite{FinisLapidContinuityGeometric} the authors extended the Arthur-Selberg trace formula to obtain the equality
\begin{align} \label{Arthur}
\sum_{\mathfrak{o}} J_{\mathfrak{o}}(h)
=\sum_{[P]} \frac{1}{|W(M_P)|} \sum_{s \in W(M_P)} \iota_s \int_{i\mathfrak{a}_{L_s}}\mathrm{tr}(\mathcal{M}_L(P,\lambda)M(P,s)\rho(P,\lambda,h)d\lambda
\end{align}
together with the absolute convergence of the sum over $\mathfrak{o}$ and the integral over $i\mathfrak{a}_{L_s}$.  Here the sum is over all association classes of standard parabolic subgroups of $G$, including $G$ itself.  Thus the absolute convergence statements in the theorem follow.

Provided that the measure on $A$ is chosen appropriately the contribution of $[P]=G$ to the spectral side of \eqref{Arthur} here is just
$$
\sum_{\pi} \frac{1}{2\pi i}\int_{\mathrm{Re}(s)=0}\mathrm{tr}\,\pi_s(h) ds,
$$
where the sum is over isomorphism classes of cuspidal automorphic representations of $A \backslash G(\A_F)$.  Thus pulling the contribution of the $[P]=G$ summand to one side we see that the quantity \eqref{to:show} is equal to 
\begin{align} \label{f1f2}
\sum_{\pi} \frac{1}{2\pi i}\int_{\mathrm{Re}(s)=0}\left(\mathrm{tr}\,\pi_s(f_1)-\mathrm{tr}\,\pi^\vee_s(f_2) \right)ds
\end{align}
where the sums are over isomorphism classes of cuspidal automorphic representations of $A \backslash G(\A_F)$.  We now note that $\omega_{r(\pi)}(aI_n)=\omega_{\pi}(a^NI_m)$. It follows from this and our choice of $d$ that \eqref{f1f2} is equal to \eqref{spec:side2}, proving the theorem.
\end{proof}

In the remaining pages of this paper we prove that \eqref{spec:side2} is equal to \eqref{res:sum} assuming a special case of Langlands' functoriality.  Before we do this we emphasize again that it makes perfect sense to study \eqref{spec:side2} using \eqref{to:show} without assuming any conjecture, and the author hopes that some progress towards the conjecture we are about to state can be made by proceeding in this manner.  

Here is the conjecture we will invoke:

\begin{conj}[Langlands] \label{conj:Langl} Let $\psi:F \backslash \A_F \to \CC^\times$ be a nontrivial character, and choose $d^\infty \in \A_F^{\infty \times }$ such that $x \mapsto \psi(d^\infty x)$ is unramified at every finite place.  For each everywhere unramified cuspidal automorphic representation $\pi$ of $G(\A_F)^1$ the function  
$$
L(s,\pi^\infty,r)
$$
admits a meromorphic continuation to the plane, holomorphic except for a possible pole at $s=1$.  It satisfies the functional equation
$$
L(s,\pi^\infty,r)=\omega_{r(\pi)}(d^\infty)^{-1}|d^\infty |^{n(1-s)}\gamma(s,\pi_\infty,r,\psi_\infty)L(1-s,\pi^{\infty \vee},r).
$$ 
\end{conj}

In stating the conjecture in this manner we are using the fact that
\begin{align*}
\varepsilon(s,\pi^\infty,r,\psi^\infty):&=\varepsilon(s,r(\pi^\infty),\psi^{\infty})\\&=\omega_{r(\pi)}(d^\infty)^{-1}|d^\infty|^{n(1-s)}\varepsilon(s,r(\pi^\infty),\psi^\infty(d^\infty \cdot))=\omega_{r(\pi)}(d^\infty)^{-1}|d^\infty|^{n(1-s)}
\end{align*}
(compare \cite[(3.2.3)]{Tate_NT})  where $r(\pi^\infty)$ is the transfer of $\pi^\infty$ to $\GL_m(\A_F^\infty)$.  It is known to exist as an admissible representation.

\begin{thm} \label{thm:reint}
  If Conjecture \ref{conj:Langl} is true for all unramified cuspidal automorphic representations of $A \backslash G(\A_F)$ then \eqref{spec:side2} is equal 
to the absolutely convergent sum
\begin{align} \label{res}
\sum_{\pi} \mathrm{Res}_{s=1}\mathrm{tr}\,\pi_{\infty s}(f)L(s,\pi^\infty,r).
\end{align}
\end{thm}
 
\begin{rem}
We remark that the proof of Theorem \ref{thm:reint} requires Theorem \ref{thm-lfe} in particular at the archimedean places; working outside a finite set of places including the archimedean ones where one can assume an unramified functional equation is not enough.
\end{rem}

For an admissible irreducible representation $\pi$ of $G(\A_F)$ let $\mathcal{C}(\pi,\mathrm{Re}(s))$ be its analytic conductor as defined by Iwaniec and Sarnak (\cite[\S 1]{BrumleyNarrow} is a nice reference).  We have the following corollary of Conjecture \ref{conj:Langl}:
\begin{cor} \label{cor-Langl} Assume Conjecture \ref{conj:Langl} for the unramified cuspidal automorphic representation $\pi$.  For any real numbers
 $A<B$ and $A \leq \mathrm{Re}(s) \leq B$ one has an estimate
$$
(s-1)^{\mathrm{ord}_{s=1}L(s,\pi^\infty,r)}L(s,\pi^\infty,r) \ll_{A,B,M} \mathcal{C}(\pi,\mathrm{Im}(s))^M
$$
for some $M>0$. 
\end{cor}

\begin{proof}
Notice that $\mathcal{C}(r(\pi),\mathrm{Im}(s)) \ll_N \mathcal{C}(\pi,\mathrm{Im}(s))^N$ for some integer $N$ depending only on $r$.  Thus, since we are assuming Conjecture \ref{conj:Langl}, to prove the corollary it suffices to prove that 
$$
(s-1)^{\mathrm{ord}_{s=1}L(s,\pi^\infty,r)}L(s,\pi^\infty,r)  \ll_{A,B,M} \mathcal{C}(r(\pi),\mathrm{Im}(s))^M.
$$
This is a standard preconvexity estimate, see, e.g., \cite[(10)]{BrumleyNarrow}. 
\end{proof}
\begin{proof}[Proof of Theorem \ref{thm:reint}]

Since $f \in C_c^\infty(G(F_\infty))$ the trace $\mathrm{tr}\,\pi_{\infty s}(f)$ is entire as a function of $s$.  We also note that if $A\leq \mathrm{Re}(s) \leq B$ then $\mathrm{tr}\,\pi_{\infty s}(f)$ is rapidly decreasing as a function of $\mathcal{C}(\pi,\mathrm{Im}(s))$ (see \cite[Lemma 4.4]{GetzApproach}).  Thus for $\sigma$ sufficiently large 
\begin{align} \label{bound-2}
\sum_{\pi} \int_{\mathrm{Re}(s)=\sigma}\left|\mathrm{tr}\,\pi_{s}(f\LL_{r})\right| ds<\infty.
\end{align}
  Applying Corollary \ref{cor-Langl} we see that \eqref{res} also converges absolutely.  

Now consider 
\begin{align*}
\sum_{\pi} \frac{1}{2\pi i}\int_{\mathrm{Re}(s)=\sigma} \frac{|d^\infty|^{n(s-1)}}{\omega_{r(\pi)}(d^\infty)}\mathrm{tr}\,(\pi^\vee)_{s}(\mathcal{F}_{r,\psi}(f) \LL_r)ds.  
\end{align*}
By Theorem \ref{thm-lfe2} and Conjecture \ref{conj:Langl} this is equal to 
\begin{align} \label{apply-conj}
\nonumber \sum_{\pi} \frac{1}{2\pi i}\int_{\mathrm{Re}(s)=\sigma} &\frac{|d^\infty|^{n(s-1)}}{\omega_{r(\pi)}(d^\infty )}\gamma(1-s,\pi_\infty,r,\psi)\mathrm{tr}\,\pi_{\infty 1-s}(f)
\mathrm{tr}\,(\pi^{\infty \vee})_s(\LL_r)ds\\
\nonumber &=  
\sum_{\pi} \frac{1}{2\pi i}\int_{\mathrm{Re}(s)=\sigma} \mathrm{tr}\,\pi_{\infty 1-s}(f)
L(1-s,\pi^\infty,r)ds\\
&=\sum_{\pi} \frac{1}{2\pi i}\int_{\mathrm{Re}(s)=1-\sigma} \mathrm{tr}\,\pi_{\infty s}(f)
L(s,\pi^\infty,r)ds.
\end{align}
Applying Corollary \ref{cor-Langl} and \eqref{bound-2} we deduce that this converges absolutely.  We now shift the contour in \eqref{apply-conj} to the line $\mathrm{Re}(s)=\sigma$, picking up the contribution of \eqref{res} from the poles at $s=1$, and deduce the theorem.
\end{proof}

\section*{Acknowledgements}

The author thanks W.~Casselman and F.~Shahidi for useful conversations.  In particular, the observation that the conjecture of Braverman and Kazhdan considered in this paper ought to follow from Paley-Wiener theorems on reductive groups was made by Casselman in an email to the author.  Moreover, W-W. Li and Shahidi deserve thanks for their comments on an earlier version of this paper.
The author appreciates S.~Cheng's help with references, D.~Barbasch and L.~Saper's aid with questions, and H.~Hahn's constant encouragement and help with editing.  V.~Lafforgue also provided many useful suggestions on a previous version of this paper which improved the accuracy and exposition significantly.


\bibliography{../refs}{}
\bibliographystyle{alpha}

\end{document}